\documentclass[12pt]{article}
\usepackage[english]{babel}
\usepackage{amssymb, amsfonts, amsthm, amsmath}

\usepackage{mathrsfs}
\usepackage{graphicx}
\usepackage{wrapfig}
\usepackage{caption2}
\usepackage{xcolor}

\newtheorem{theorem}{\bf Theorem}
\newtheorem{Remark}{\bf Remark}
\newtheorem{lemma}{\bf Lemma}

\begin{document}
\title{\textbf{\large Existence and uniqueness theorems for one class of
Hammerstein-type nonlinear integral equations}}
	
	\author{Zahra Keyshams, Khachatur A. Khachatryan, Monire Mikaeili Nia}
	\date{}
    \maketitle
\textbf{Abstract:} The class of nonlinear integral equations on the positive half-line with a monotone operator of
Hammerstein type is studied. With various partial representations of the corresponding kernel and
nonlinearity, this class of equations has applications in the dynamic theory of $p$-adic strings, in the
kinetic theory of gases, in the theory of radiation transfer and in the mathematical theory of the
geographical spread of epidemic diseases. A constructive theorem for the existence of a nontrivial
bounded solution is proved. The asymptotic behavior of the constructed solution at infinity is studied.
We also prove a theorem for the uniqueness of a solution in the class of nonnegative nontrivial and
bounded functions. At the end of the work, specific particular examples of the kernel and nonlinearity
of this class of equations are given, which are of independent interest.


\textbf{Bibliography:} 30 titles.


\section{Introduction}

This work is devoted to the study of questions of existence, uniqueness, as well as the construction of an approximate solution to the following class of nonlinear integral equations on the positive half-line with an integral operator of Hammerstein type:
\begin{equation}\label{Khachatryan1}
f(x)=\int\limits_0^\infty K(x,t)G(f(t))dt,\quad x\in\mathbb{R}^+:=[0,+\infty)
\end{equation}
with respect to the desired measurable, non-negative and bounded function $f$ on $\mathbb{R}^+.$ In equations \eqref{Khachatryan1} the kernel $K$ is defined on the set $\mathbb{R}^+\times\mathbb{R}^+$ and satisfies the following basic conditions:
\begin{enumerate}
  \item [a)] $K(x,t)>0,\quad (x,t)\in \mathbb{R}^+\times\mathbb{R}^+,\quad \sup\limits_{x\in\mathbb{R}^+}\int\limits_0^\infty K(x,t)dt=1,$
  \item [b)] $\gamma(x):=1-\int\limits_0^\infty K(x,t)dt\not\equiv0,\quad x\in\mathbb{R}^+,\quad \gamma\in L_1^0(\mathbb{R}^+),$
\end{enumerate}
where $L_1^0(\mathbb{R}^+)$ is the space of integrable functions on the set $\mathbb{R}^+$ having zero limit at infinity,
\begin{enumerate}
	\item [c)] $K\in C(\mathbb{R}^+\times\mathbb{R}^+)$ and there exist $\lambda^*(t),$ $K^*(\tau)$ with properties
	\begin{enumerate}
		\item [$c_1)$] $\lambda^*(t)>1,\quad t\in\mathbb{R}^+,\quad \lambda^*-1\in L_1^0(\mathbb{R}^+),$
		\item [$c_2)$] $K^*(\tau)>0,\,\, \tau\in\mathbb{R}:=(-\infty,+\infty),\,\, K^*\in L_1(\mathbb{R})\cap L_\infty(\mathbb{R}),\,\, \int\limits_{-\infty}^\infty |t|K^{*}(t)dt<+\infty$ such that
 $$K(x,t)\leq \lambda^*(t)K^*(x-t),\quad (x,t)\in \mathbb{R}^+\times\mathbb{R}^+.$$
\end{enumerate}
\end{enumerate}
The nonlinear function $G(u)$ is defined on the set $\mathbb{R}^+$ and has the following properties
\begin{enumerate}
  \item [1)] $G\in C(\mathbb{R}^+),\,\, y=G(u)$ increases monotonically on $\mathbb{R}^+,$
  \item [2)] $G(0)=0$ and there is a number $\eta>0$ such that $G(\eta)=\eta,$
  \item [3)] $y=G(u)$ is a concave function on the set $\mathbb{R}^+,$
  \item [4)] there exists a number $\alpha\in(0,1)$ such that for all $\sigma\in(0,1)$ and $u\in[0,\eta],$ the following inequality holds $$G(\sigma u)\geq \sigma^\alpha G(u).$$
\end{enumerate}
Equation \eqref{Khachatryan1}, with different representations of the kernel $K$ and nonlinear function $G,$ arises in many areas
of mathematical science. In particular, the equations of this nature are found in the dynamic theory
of $p$-adic open-closed strings for the scalar field of tachyons, in the kinetic theory of gases and
plasmas within the framework of a modified model of the Boltzmann equation, in the nonlinear theory
of radiative transfer in inhomogeneous media and in the mathematical theory of the geographical
spread of epidemic diseases within the framework of the Diekmann-Kaper model (see \cite{aref1}-\cite{diek9}). It should also be noted that some discrete analogs of equation \eqref{Khachatryan1} find applications in radio engineering (see \cite{bas10}, \cite{su11}).

In \cite{arb15}, \cite{arb12}, \cite{diek9}, \cite{khach16} and \cite{khach14}-\cite{khach13},  the equation \eqref{Khachatryan1} has been studied when the kernel $K$ has two properties of being dependent on the difference of its arguments and being conservative, with various conditions on $G.$ 
It is worth noticing that the kernel $K$ is called conservative if\\
 \[ K(x) > 0, \quad x \in \mathbb{R}, \quad \text{and} \quad \int^{\infty}_{-\infty} K(x) \, dx = 1.\] For kernels of the form $K(x,t)=\tilde{K}(x-t)-\tilde{K}(x+t)$, $(x,t)\in\mathbb{R}^+\times\mathbb{R}^+$, where $\tilde{K}(\tau)=\tilde{K}(-\tau)\geq0$, $\tau\in\mathbb{R}^+$, $\tilde{K}(\tau) \downarrow$ on $\mathbb{R}^+$, $\tilde{K}\in L_1(\mathbb{R})\cap L_\infty(\mathbb{R})$, $\int\limits_0^\infty \tilde{K}(\tau)d\tau=\frac{1}{2}$, $\int\limits_0^\infty \tau\tilde{K}(\tau)d\tau<+\infty$,  equation \eqref{Khachatryan1} was studied in \cite{vla2}, \cite{vla3} and \cite{khach17}-\cite{petr20}. In these studies, questions of existence, uniqueness, and asymptotic behavior on $+\infty$ of a nontrivial solution, that is nonnegative on $\mathbb{R}^+$, as well as various function spaces were mainly studied. It should be noted that nonlinear integral equations of a similar structure are found in studies of various problems for nonlinear wave equations and equations such as heat conductivity (see.\cite{Ruzhansky27}-\cite{Ruzhansky30}).

In the present work, under conditions $a)-c)$ and $1)-4),$ we will first deal with the question of constructing a nontrivial nonnegative continuous and bounded solution on $\mathbb{R}^+$. Then, we will study the asymptotic behavior at infinity of the constructed solution. After which, with an additional restriction - $K(x,t)=K(t,x),$ $(x,t)\in\mathbb{R}^+\times\mathbb{R}^+$, we
prove the uniqueness of the solution in the class of non-negative nontrivial bounded functions  on $\mathbb{R}^+$. Note that the proof of the existence theorem is constructive in nature, since in the course of constructing a limited solution to equation \eqref{Khachatryan1}, special successive approximations are given which uniformly converge to the solution of equation \eqref{Khachatryan1}.  Moreover, it turns out that in the numerical estimate for the
difference between successive approximations, the right side tends to
zero as a geometric progression. Moreover, note  that in the proof of the uniqueness of the
solution, certain geometric inequalities for concave functions and a priori integral estimates for
arbitrary non-trivial non-negative and bounded solutions are obtained. Using the results obtained from equation \eqref{Khachatryan1}, it is also possible to investigate the existence of a non-negative, non-trivial, integrable and bounded  solution on $\mathbb{R}^+$ to one class of non-linear integral equations with a monotone operator
of the Hammerstein-Nemytski type. At the end of this work, specific particular examples of the kernel $K$ and nonlinear function $G$ are given to illustrate the importance of the obtained results.

\section{Auxiliary Lemmas}

The next lemma is proved by similar reasoning as Lemmas~1 and 2 from \cite{petr20}.
\begin{lemma}
Under conditions $a)-c)$ and $1)-3)$ any non-negative and bounded function $f$ on $\mathbb{R}^{+}$  which is the solution of equation \eqref{Khachatryan1}  has the following properties
$$
I)\quad f\in C(\mathbb{R}^+),\quad II)\quad f(x)\leq\eta,\quad x\in\mathbb{R}^+.
$$
\end{lemma}
The following key lemma also holds.
\begin{lemma}
Under conditions $a)-c), 1)-3),$ if the kernel $K(x,t)$ also has the symmetry property:
\begin{equation}\label{Khachatryan2}
K(x,t)=K(t,x),\quad (x,t)\in\mathbb{R}^+\times\mathbb{R}^+,
\end{equation}
then for any non-negative and bounded solution $f$ on $\mathbb{R}^+$ the following inclusion holds:
\begin{equation}\label{Khachatryan3}
G\circ f-f\in L_1(\mathbb{R}^+).
\end{equation}
\end{lemma}
\begin{proof}
Let $R>0$ be an arbitrary number. Then, using conditions $a)-c),$ $ 1)-3),$ \eqref{Khachatryan2}, as well as Lemma~1, from \eqref{Khachatryan1} by virtue of Fubini’s theorem (see \cite{kol21}) we have
\begin{align*}
&0\leq \int\limits_0^R(\eta-f(x))dx=\eta \int\limits_0^R\gamma(x)dx+\int\limits_0^R\int\limits_0^\infty (\eta-G(f(t)))K(x,t)dtdx\\
&\leq \eta \int\limits_0^\infty\gamma(x)dx+\int\limits_0^R\int\limits_0^R (\eta-G(f(t)))K(x,t)dtdx+ \int\limits_0^R\int\limits_R^\infty (\eta-G(f(t)))K(x,t)dtdx \\
&\leq \eta \int\limits_0^\infty\gamma(x)dx+\int\limits_0^R(\eta-G(f(t)))\int\limits_0^\infty K(t,x)dxdt+ \int\limits_0^R\int\limits_R^\infty (\eta-G(f(t)))K(x,t)dtdx\\
&\leq \eta \int\limits_0^\infty\gamma(x)dx+\int\limits_0^R(\eta-G(f(t)))dt+ \int\limits_0^R\int\limits_R^\infty (\eta-G(f(t)))K^*(x-t)\lambda^*(t)dtdx\\
&\leq \eta \int\limits_0^\infty\gamma(x)dx+\int\limits_0^R(\eta-G(f(t)))dt+ \eta\int\limits_0^R\int\limits_R^\infty K^*(x-t)\lambda^*(t)dtdx\\
&\small{= \eta \int\limits_0^\infty\gamma(x)dx+\int\limits_0^R(\eta-G(f(t)))dt+ \eta\int\limits_0^R\int\limits_R^\infty K^*(x-t)(\lambda^*(t)-1)dtdx+ \eta\int\limits_0^R\int\limits_R^\infty K^*(x-t)dtdx}\\
&\leq\eta \int\limits_0^\infty\gamma(x)dx+\int\limits_0^R(\eta-G(f(t)))dt+ \eta\int\limits_0^R(\lambda^*(t)-1)dt \int\limits_{-\infty}^{\infty}K^*(y)dy+ \eta\int\limits_0^R\int\limits_{R-x}^\infty K^*(-y)dydx\\
&=\eta \int\limits_0^\infty\gamma(x)dx+\int\limits_0^R(\eta-G(f(t)))dt+ \eta\int\limits_0^\infty(\lambda^*(t)-1)dt \int\limits_{-\infty}^{\infty}K^*(y)dy+ \eta\int\limits_0^R\int\limits_{z}^\infty K^*(-y)dydx\\
&\leq \eta \int\limits_0^\infty\gamma(x)dx+\int\limits_0^R(\eta-G(f(t)))dt+ \eta\int\limits_0^\infty(\lambda^*(t)-1)dt \int\limits_{-\infty}^{\infty}K^*(y)dy+ \eta\int\limits_0^\infty\int\limits_{z}^\infty K^*(-y)dydx\\
&=\eta \int\limits_0^\infty\gamma(x)dx+\int\limits_0^R(\eta-G(f(t)))dt+ \eta\int\limits_0^\infty(\lambda^*(t)-1)dt \int\limits_{-\infty}^{\infty}K^*(y)dy+ \eta\int\limits_0^\infty K^*(-y)ydy\\
&\leq \eta \int\limits_0^\infty\gamma(x)dx+\int\limits_0^R(\eta-G(f(t)))dt+ \eta\int\limits_0^\infty(\lambda^*(t)-1)dt \int\limits_{-\infty}^{\infty}K^*(y)dy+ \eta\int\limits_{-\infty}^\infty |y| K^*(y)dy,
\end{align*}
whence it follows that
\begin{align*}
&\int\limits_0^R(G(f(x))-f(x))dx\leq \eta \int\limits_0^\infty\gamma(x)dx+ \eta\int\limits_0^\infty(\lambda^*(t)-1)dt \int\limits_{-\infty}^{\infty}K^*(y)dy+ \eta\int\limits_{-\infty}^\infty |y| K^*(y)dy.
\end{align*}
Tending the number $R\rightarrow+\infty$ to infinity in the last inequality, we get
\begin{equation}\label{Khachatryan4}
\int\limits_0^\infty(G(f(x))-f(x))dx\leq\eta\int\limits_0^\infty\gamma(x)dx+ \eta\int\limits_0^\infty(\lambda^*(t)-1)dt\int\limits_{-\infty}^{\infty}K^*(y)dy+ \eta\int\limits_{-\infty}^\infty|y|K^*(y)dy.
\end{equation}
In addition, since $0\leq f(x)\leq\eta,$ $x\in\mathbb{R}^+,$ it immediately follows from properties $1)-3)$ that
\begin{equation}\label{Khachatryan5}
G(f(x))\geq f(x),\quad x\in\mathbb{R}^+.
\end{equation}
From \eqref{Khachatryan4} and \eqref{Khachatryan5}, we come to the completion of the proof.
\end{proof}
Below, using Lemmas~1 and 2, under one additional condition on the function $f(x)$ we prove that $\eta-f\in L_1^0(\mathbb{R}^+).$

The following Lemma is true
\begin{lemma}
Under the conditions of Lemma~2, if there is a number $r>0$ such that \linebreak$\inf\limits_{x\geq r}f(x)>0,$ then $\eta-f\in L_1^0(\mathbb{R}^+).$
\end{lemma}
\begin{proof}
 First we prove that $\eta-f\in L_1(\mathbb{R}^+).$ Since $f\in C(\mathbb{R}^+)$ (see Lemma~1), then it is enough to prove that $\eta-f\in L_1(r,+\infty).$ Let us denote by $\varepsilon:=\inf\limits_{x\geq r}f(x)>0.$ Then, taking into account properties $1)-3)$ of the nonlinearity of $G,$ we will
have (see Fig. 1):
 $$\eta-G(f(x))\leq \frac{\eta-G(\varepsilon)}{\eta-\varepsilon}(\eta-f(x)),\quad x\geq r,$$
 from which it immediately follows that
\begin{equation}\label{Khachatryan6}
G(f(x))-f(x)\geq \frac{G(\varepsilon)-\varepsilon}{\eta-\varepsilon}(\eta-f(x)),\quad x\geq r
\end{equation}
Taking into account inequalities \eqref{Khachatryan4}, \eqref{Khachatryan5} and \eqref{Khachatryan6} we obtain
\begin{align*}
&	\frac{G(\varepsilon)-\varepsilon}{\eta-\varepsilon}\int\limits_r^\infty (\eta-f(x))dx\\ &\le\eta\int\limits_0^\infty\gamma(x)dx+ \eta\int\limits_0^\infty(\lambda^*(t)-1)dt\int\limits_{-\infty}^{\infty}K^*(y)dy+ \eta\int\limits_{-\infty}^\infty|y|K^*(y)dy.
	\end{align*}
Since $f(x)\leq\eta$ and $f(x)\not\equiv\eta$ (since $\gamma(x)\not\equiv0$), then $\varepsilon\in(0,\eta).$ Furthermore, from the nonlinearity properties of $G$ we
obtain that $G(\varepsilon)>\varepsilon.$ Thus, from the integral inequality obtained above we come to the estimate:
\begin{equation}\label{Khachatryan7}
\begin{array}{c}
\displaystyle\int\limits_r^\infty (\eta-f(x))dx\leq \\ \displaystyle\leq\frac{(\eta-\varepsilon)\eta}{G(\varepsilon)-\varepsilon}\left(\int\limits_0^\infty\gamma(x)dx+ \int\limits_0^\infty(\lambda^*(t)-1)dt\int\limits_{-\infty}^{\infty}K^*(y)dy+ \int\limits_{-\infty}^\infty|y|K^*(y)dy\right),
\end{array}
\end{equation}
whence it follows that $\eta-f\in L_1(r,+\infty).$ Let us now make sure that $\lim\limits_{x\rightarrow+\infty}f(x)=\eta.$ Indeed, taking into account conditions $a)-c),$ $1)-3),$ inequality \eqref{Khachatryan5}, from equation \eqref{Khachatryan1} we arrive
to the following simple estimates:
\begin{equation}\label{Khachatryan8}
\begin{array}{c}
\displaystyle 0\leq\eta-f(x)\leq \eta\gamma(x)+\int\limits_0^\infty K(x,t)(\eta-f(t))dt\\
\displaystyle\leq \eta\gamma(x)+\int\limits_0^\infty\lambda^*(t)K^*(x-t)(\eta-f(t))dt \\
\displaystyle \leq \eta\gamma(x)+\eta\int\limits_0^\infty K^*(x-t)(\lambda^*(t)-1)dt+\int\limits_0^\infty K^*(x-t)(\eta-f(t))dt\\
\displaystyle=\eta\gamma(x)+\eta\int\limits_0^\infty K^*(x-t)(\lambda^*(t)-1)dt+\int\limits_{-\infty}^\infty K^*(x-t)\psi(t)dt,\quad x\in\mathbb{R}^+,
\end{array}
\end{equation}
where
\begin{equation}\label{Khachatryan9}
\psi(t):=\left\{
           \begin{array}{ll}
             \eta-f(t), & t\in\mathbb{R}^+, \\
             0, & t<0.
           \end{array}
         \right.
\end{equation}
Since $\lim\limits_{t\rightarrow+\infty}\lambda^*(t)=1,$ $K^*\in L_1(\mathbb{R}),$ then due to the known limit relation in convolution operations (see \cite{gev22})
 we will have
\begin{equation}\label{Khachatryan10}
\lim\limits_{x\rightarrow+\infty}\int\limits_0^\infty K^*(x-t)(\lambda^*(t)-1)dt= \lim\limits_{t\rightarrow+\infty}(\lambda^*(t)-1)\int\limits_{-\infty}^\infty K^*(y)dy=0.
\end{equation}
On the other hand, since $K^*\in L_1(\mathbb{R})\cap L_\infty(\mathbb{R}),$ $\psi\in L_1(\mathbb{R})\cap L_\infty(\mathbb{R}),$ then according to Lemma~5 from \cite{arb23} we have
\begin{equation}\label{Khachatryan11}
\lim\limits_{x\rightarrow+\infty}\int\limits_{-\infty}^\infty K^*(x-t)\psi(t)dt=0.
\end{equation}
Since $\lim\limits_{x\rightarrow+\infty}\gamma(x)=0,$ then from \eqref{Khachatryan8} due to \eqref{Khachatryan10} and \eqref{Khachatryan11} we obtain that $\lim\limits_{x\rightarrow+\infty}f(x)=\eta.$ Thus the lemma is proved.
\end{proof}

\begin{figure}[h!]
  \centering
  \includegraphics[width=5in]{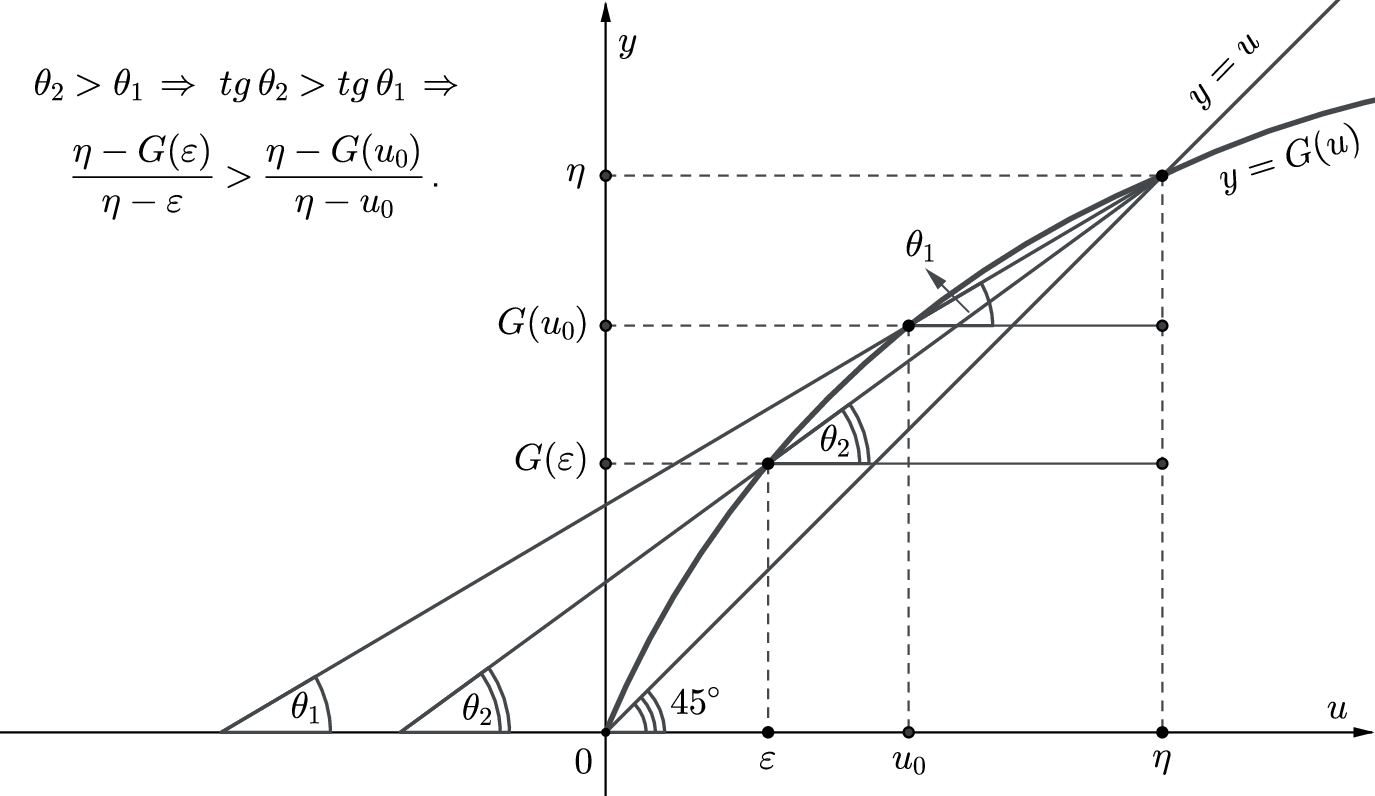}\\
  \caption{}
\end{figure}

\section{The existence of a bounded solution to the equation \eqref{Khachatryan1}}
We have the following theorem
\begin{theorem}
	Equation \eqref{Khachatryan1} under conditions $a)-c)$ and $1)-4)$  has a positive bounded continuous solution $f^*(x)$  on $\mathbb{R}^+$  and $\eta-f^*\in L_1^0(\mathbb{R}^+).$ Moreover successive approximations
	\begin{equation}\label{Khachatryan12}
		\begin{array}{c}
			\displaystyle f_{n+1}(x)=\int\limits_0^\infty K(x,t)G(f_n(t))dt,\quad x\in\mathbb{R}^+,\\
			\displaystyle f_0(x)\equiv\eta,\quad n=0,1,2,\ldots
		\end{array}
	\end{equation}
	uniformly converge to the solution $f^*(x).$ In addition, there exists a number $\sigma_0\in(0,1)$ such that
	\begin{equation}\label{Khachatryan13}
		|f_n(x)-f_{n+1}(x)|\leq \eta\alpha^{n-1}\ln\frac{1}{\sigma_0},\quad n=1,2,\ldots,\quad x\in\mathbb{R}^+,
	\end{equation}
	where the number $\alpha\in(0,1)$ is defined in condition $4).$
\end{theorem}
\begin{proof}
	First, note that by induction on $n$ it is easy to verify the following facts:
	\begin{equation}\label{Khachatryan14}
		f_n(x)\downarrow\quad\mbox{on}\quad n,\quad x\in\mathbb{R}^+
	\end{equation}
	\begin{equation}\label{Khachatryan15}
		f_n\in C(\mathbb{R}^+),\quad n=0,1,2,\ldots.
	\end{equation}
	The above statements \eqref{Khachatryan14} and \eqref{Khachatryan15} are easily obtained if conditions $a)-c)$ and $1), 2)$ are taken into
	account. Now we can verify by induction that
	\begin{equation}\label{Khachatryan16}
		\lim\limits_{x\rightarrow+\infty}f_n(x)=\eta,\quad n=0,1,2,\ldots.
	\end{equation}
	In case when $n=0,$ limit relation  \eqref{Khachatryan16} immediately follows from the definition of the zero approximation in iterations of \eqref{Khachatryan12}. Suppose that for some natural $n,$  we have $\lim\limits_{x\rightarrow+\infty}f_n(x)=\eta.$ Then, by virtue of \eqref{Khachatryan14}, \eqref{Khachatryan10} and conditions $a)-c),$ $1), \quad \text{and} \quad 2)$ from \eqref{Khachatryan12} we will have
\begin{align*}
		& 0\leq\eta-f_{n+1}(x)\leq \eta\gamma(x)+\int\limits_0^\infty K(x,t)(\eta-f_n(t))dt\\ &\leq\eta\gamma(x)+\int\limits_0^\infty\lambda^*(t)K^*(x-t)(\eta-f_n(t))dt \\
		& \leq \eta\gamma(x)+\eta\int\limits_0^\infty K^*(x-t)(\lambda^*(t)-1)dt+\int\limits_0^\infty K^*(x-t)(\eta-f_n(t))dt\rightarrow0,\,\,\,\, x\rightarrow+\infty,
\end{align*}
	whence it follows that $\lim\limits_{x\rightarrow+\infty}f_{n+1}(x)=\eta.$ Hence, the relation \eqref{Khachatryan16} is proved.
	
	Since $K(x,t)>0,$ $(x,t)\in\mathbb{R}^+\times\mathbb{R}^+,$ $G(0)=0$ and $y=G(u) \uparrow$ on $\mathbb{R}^+,$ one can easily verify by mathematical induction that, 
	\begin{equation}\label{Khachatryan17}
		f_n(x)>0,\quad n=0,1,2,\ldots,\quad x\in\mathbb{R}^+.
	\end{equation}
	Consider the function
	\begin{equation}\label{Khachatryan18}
		B(x):=\frac{f_2(x)}{f_1(x)},\quad x\in\mathbb{R}^+.
	\end{equation}
	Firstly, from \eqref{Khachatryan14}, \eqref{Khachatryan15} and \eqref{Khachatryan17} it immediately follows that
	\begin{equation}\label{Khachatryan19}
		B\in C(\mathbb{R}^+),
	\end{equation}
	\begin{equation}\label{Khachatryan20}
		0<B(x)\leq1, \quad x\in\mathbb{R}^+.
	\end{equation}
	On the other hand, due to \eqref{Khachatryan16} we have
	\begin{equation}\label{Khachatryan21}
		\lim\limits_{x\rightarrow+\infty} B(x)=\frac{f_2(+\infty)}{f_1(+\infty)}=1.
	\end{equation}
	From \eqref{Khachatryan20} and \eqref{Khachatryan21} we conclude that there is a number $r_0>0$ such that for $x\geq r_0$ the following inequality from below holds:
	\begin{equation}\label{Khachatryan22}
		B(x)\geq \frac{1}{2}.
	\end{equation}
	On the interval $[0,r_0]$ for the function $B(x)$ all conditions of the Weierstrass theorem are satisfied.  Therefore, taking into account \eqref{Khachatryan20} we can say that there is a point $x_0\in[0,r_0]$ such that
	\begin{equation}\label{Khachatryan23}
		\min\limits_{x\in[0,r_0]}B(x)=B(x_0)>0.
	\end{equation}
	Taking into account \eqref{Khachatryan22} and \eqref{Khachatryan23} we conclude that
	\begin{equation}\label{Khachatryan24}
		\sigma_0:=\inf\limits_{x\in\mathbb{R}^+}B(x)\geq \min\left\{\frac{1}{2}, B(x_0)\right\}>0.
	\end{equation}
	So from \eqref{Khachatryan14}, \eqref{Khachatryan20} and \eqref{Khachatryan24} we obtain the following inequalities
	\begin{equation}\label{Khachatryan25}
		\sigma_0 f_1(t)\leq f_2(t)\leq f_1(t),\quad t\in \mathbb{R}^+.
	\end{equation}
	From \eqref{Khachatryan25} due to condition $4),$ it follows that
	\begin{equation}\label{Khachatryan26}
		\sigma_0^\alpha G(f_1(t))\leq G(\sigma_0 f_1(t))\leq G(f_2(t))\leq G(f_1(t)),\quad t\in \mathbb{R}^+.
	\end{equation}
	Let us multiply both sides of \eqref{Khachatryan26} by the function $K(x,t)$ and integrate the resulting inequality over $t$ in the range from $0$ to $+\infty.$ So by virtue of \eqref{Khachatryan12}, we obtain
	\begin{equation}\label{Khachatryan27}
		\sigma_0^\alpha f_2(x)\leq f_3(x)\leq f_2(x),\quad x\in\mathbb{R}^+.
	\end{equation}
	From \eqref{Khachatryan27} and condition $4),$ it follows that
	\begin{equation}\label{Khachatryan28}
		\sigma_0^{\alpha^2} G(f_2(t))\leq G(\sigma_0^\alpha f_2(t))\leq G(f_3(t))\leq G(f_2(t)),\quad t\in \mathbb{R}^+.
	\end{equation}
	Repeating the above process again we get
	$$\sigma_0^{\alpha^2} f_3(x)\leq f_4(x)\leq f_3(x),\quad x\in\mathbb{R}^+.$$
	Continuing these arguments for an arbitrary natural number $n,$ we obtain the following inequalities
	$$\sigma_0^{\alpha^{n-1}} f_n(x)\leq f_{n+1}(x)\leq f_n(x),\quad x\in\mathbb{R}^+.$$
	Hence using \eqref{Khachatryan14}, we obtain
	\begin{equation}\label{Khachatryan29}
		0\leq f_n(x)-f_{n+1}(x)\leq (1-\sigma_0^{\alpha^{n-1}})f_n(x)\leq \eta (1-\sigma_0^{\alpha^{n-1}}),\, n=1,2,\ldots,\, x\in\mathbb{R}^+.
	\end{equation}
	Since $\lim\limits_{n\rightarrow\infty}\sigma_0^{\alpha^{n-1}}=1,$ then from \eqref{Khachatryan29} it follows the uniform convergence of the sequence of continuous functions $\{f_n(x)\}_{n=1}^\infty$  on $\mathbb{R}^+$ 
	\begin{equation}\label{Khachatryan30}
		\lim\limits_{n\rightarrow\infty}f_n(x)=f^*(x),
	\end{equation}
	and
	\begin{equation}\label{Khachatryan31}
		0\leq f^*(x)\leq\eta,\quad x\in\mathbb{R}^+.
	\end{equation}
	Using B. Levy's limit theorem (see \cite{kol21}) due to \eqref{Khachatryan14}, \eqref{Khachatryan17} and conditions $a)-c),$ $1)-3)$ we conclude that $f^*(x)$ satisfies the nonlinear integral equation \eqref{Khachatryan1}. From Lemma~1 it follows that
	\begin{equation}\label{Khachatryan32}
		f^*\in C(\mathbb{R}^+).
	\end{equation}
	Let us now make sure that there exists
	\begin{equation}\label{Khachatryan33}
		\lim\limits_{x\rightarrow+\infty}f^*(x)=\eta.
	\end{equation}
Based on the well-known fact from mathematical analysis, the following inequality holds
	\begin{equation}\label{Khachatryan34}
		1-\sigma_0^{\alpha^{n-1}}\leq \alpha^{n-1}\ln\frac{1}{\sigma_0},\quad n=1,2,\ldots,\quad \sigma_0\in(0,1),\quad \alpha\in(0,1).
	\end{equation}
	From \eqref{Khachatryan34} and \eqref{Khachatryan29}, according to the Weierstrass theorem on majorant number series, it follows that the functional series
	\begin{equation}\label{Khachatryan35}
		\sum\limits_{n=1}^\infty (f_n(x)-f_{n+1}(x))
	\end{equation}
	converges uniformly in $x\in\mathbb{R}^+.$ Therefore, taking into account the relations \eqref{Khachatryan16} and \eqref{Khachatryan30}, we have
	\begin{align*}
		\lim\limits_{x\rightarrow+\infty}f^*(x)&=\lim\limits_{x\rightarrow+\infty} \lim\limits_{N\rightarrow+\infty}\left(f_0(x)+\sum\limits_{n=0}^N(f_{n+1}(x)-f_n(x))\right)\\
	&=\lim\limits_{x\rightarrow+\infty} \left(f_0(x)+\sum\limits_{n=0}^\infty(f_{n+1}(x)-f_n(x))\right)\\
	&=\lim\limits_{x\rightarrow+\infty}f_0(x)+ \sum\limits_{n=0}^\infty(\lim\limits_{x\rightarrow+\infty}f_{n+1}(x)-\lim\limits_{x\rightarrow+\infty}f_n(x))\\
	&=\eta.
	\end{align*}
	Since $\gamma(x)\not\equiv0,$ $x\in\mathbb{R}^+,$ then $f(x)\not\equiv\eta.$ From \eqref{Khachatryan33} it follows that $f(x)\not\equiv0.$ Below we prove that
	\begin{equation}\label{Khachatryan36}
		0<f^*(x)<\eta,\quad x\in\mathbb{R}^+.
	\end{equation}
	Indeed, by virtue of \eqref{Khachatryan31}, \eqref{Khachatryan32}, \eqref{Khachatryan33} there is a number $\delta>0$ such that for $x\geq\delta$ the estimation from below holds:
	\begin{equation}\label{Khachatryan37}
		f^*(x)\geq\frac{\eta}{2}.
	\end{equation}
	Therefore, taking into account conditions $a), 1)$ and $2)$ from \eqref{Khachatryan1} we get
	\begin{equation}\label{Khachatryan38}
		f^*(x)\geq G\left(\frac{\eta}{2}\right)\int\limits_\delta^\infty K(x,t)dt>0,\quad x\in\mathbb{R}^+.
	\end{equation}
	On the other hand, considering conditions $a), b), 1), 2),$ relations \eqref{Khachatryan12}, and \eqref{Khachatryan13} one can see from \eqref{Khachatryan1} that
	\begin{equation}\label{Khachatryan39}
		\eta-f^*(x)\geq \int\limits_0^\infty K(x,t)(\eta-G(f^*(t)))dt,\quad x\in\mathbb{R}^+.
	\end{equation}
	Since $f^*(x)\not\equiv\eta,$ $x\in\mathbb{R}^+,$ $f^*(x)\leq\eta,$ $x\in\mathbb{R}^+,$ $f^*\in C(\mathbb{R}^+),$ then there exists a point $x_0>0$ and a number $\delta_0\in(0,x_0)$ such that
	\begin{equation}\label{Khachatryan40}
		\beta_0:=\inf\limits_{x\in(x_0-\delta_0,x_0+\delta_0)}(\eta-G(f^*(x)))>0,
	\end{equation}
	because $y=G(u)\uparrow$ on $\mathbb{R}^+,$ $G(0)=0,$ $G(\eta)=\eta,$ $G\in C(\mathbb{R}^+).$
	
	Taking into consideration \eqref{Khachatryan40}, from \eqref{Khachatryan39} we obtain
	$$\eta-f^*(x)\geq \int\limits_{x_0-\delta_0}^{x_0+\delta_0}K(x,t)(\eta-G(f^*(t)))dt\geq \beta_0 \int\limits_{x_0-\delta_0}^{x_0+\delta_0} K(x,t)dt>0.$$
	whence it follows that
	$$f^*(x)\leq\eta-\beta_0\int\limits_{x_0-\delta_0}^{x_0+\delta_0} K(x,t)dt<\eta,\quad x\in\mathbb{R}^+.$$
	Taking \eqref{Khachatryan37} into account, by Lemma~3 we also obtain that $\eta-f^*\in L_1(\mathbb{R}^+).$
	
	Thus the theorem is completely proved.
\end{proof}
\begin{Remark}
It should be noted that in the particular case when $G(u)=u^\alpha,$ \linebreak$u\in\mathbb{R}^+,$ $\alpha\in(0,1)$, and the kernel $K(x,t)$ has the form like $k(x-t)-k(x+t),$ for $(x,t)\in\mathbb{R}^+\times\mathbb{R}^+$
and represented in the form of a Gaussian distribution, the indicated construction of a nonnegative,
non-trivial and bounded on $\mathbb{R}^+$ solution of equation \eqref{Khachatryan1} was first used in the work \cite{vla3}.
\end{Remark}

\section{Uniqueness of solution of equation \eqref{Khachatryan1}}

In this section we will discuss the question of the uniqueness of the solution to the integral equation \eqref{Khachatryan1} under the additional restriction \eqref{Khachatryan2} on the kernel $K(x,t).$

The following theorem holds
\begin{theorem}
Under conditions $a)-c),$ $1)-4)$ and \eqref{Khachatryan2}, equation \eqref{Khachatryan1} has a unique solution in the class of non-negative nontrivial bounded functions on $\mathbb{R}^+.$
\end{theorem}
\begin{proof}
Let us assume that equation \eqref{Khachatryan1}, in addition to the solution $f^*(x)$ constructed using successive
approximations  \eqref{Khachatryan12}, has another non-negative and bounded on $\mathbb{R}^+$ solution $\tilde{f}(x).$ Then, by Lemma~1
\begin{equation}\label{Khachatryan41}
\tilde{f}\in C(\mathbb{R}^+).
\end{equation}
Using conditions $1)-3),$ $a), b)$ and inclusion \eqref{Khachatryan41}, as well as repeating similar reasoning as in the proof of
Theorem~1, we can be convinced that
\begin{equation}\label{Khachatryan42}
0<\tilde{f}(x)<\eta,\quad x\in\mathbb{R}^+.
\end{equation}
Taking into account \eqref{Khachatryan42} it is easy to verify by induction on $n$ that
\begin{equation}\label{Khachatryan43}
\tilde{f}(x)\leq f_n(x),\quad n=0,1,2,\ldots,\quad x\in\mathbb{R}^+.
\end{equation}
In inequality \eqref{Khachatryan43} tending the number $n$ to infinity, we obtain
\begin{equation}\label{Khachatryan44}
\tilde{f}(x)\leq f^*(x),\quad x\in\mathbb{R}^+.
\end{equation}
Our main goal is to prove that
\begin{equation}\label{Khachatryan45}
\tilde{f}(x)=f^*(x),\quad x\in\mathbb{R}^+.
\end{equation}
Suppose that there exists $x^*\in\mathbb{R}^+$ such that $\tilde{f}(x^*)\neq f^*(x^*).$ Then, taking into account \eqref{Khachatryan41} and \eqref{Khachatryan32} we can assert that
there exists a point $\tilde{x}\in(0,+\infty)$ and a number $\tilde{\delta}\in(0,\tilde{x})$ such that
\begin{equation}\label{Khachatryan46}
\tilde{f}(x)\neq f^*(x),\quad x\in(\tilde{x}-\tilde{\delta},\tilde{x}+\tilde{\delta}).
\end{equation}
From \eqref{Khachatryan44} and \eqref{Khachatryan46} it immediately follows that
\begin{equation}\label{Khachatryan47}
\tilde{f}(x)< f^*(x),\quad x\in(\tilde{x}-\tilde{\delta},\tilde{x}+\tilde{\delta}).
\end{equation}
Consider the following measurable set:
\begin{equation}\label{Khachatryan48}
\Gamma:=\{x\in\mathbb{R}^+:\tilde{f}(x)< f^*(x)\}.
\end{equation}
From \eqref{Khachatryan47} it follows that $\Gamma\neq\emptyset$ and $mes\Gamma\geq2\tilde{\delta}>0.$ Let us prove that
\begin{equation}\label{Khachatryan49}
\chi_1(x):=(G(\tilde{f}(x))-\tilde{f}(x))\int\limits_0^\infty K(x,t)(G(f^*(t))-G(\tilde{f}(t)))dt\in L_1(\mathbb{R}^+),
\end{equation}
\begin{equation}\label{Khachatryan50}
\chi_2(x):=(G(f^*(x))-G(\tilde{f}(x)))\int\limits_0^\infty K(x,t)(G(\tilde{f}(t))-\tilde{f}(t))dt\in L_1(\mathbb{R}^+).
\end{equation}
By virtue of $a),$ \eqref{Khachatryan42}, \eqref{Khachatryan36} and Lemma~2, we have
\begin{equation}\label{Khachatryan51}
0\leq \chi_1(x)\leq\eta \int\limits_0^\infty K(x,t)dt(G(\tilde{f}(x))-\tilde{f}(x))\leq \eta (G(\tilde{f}(x))-\tilde{f}(x))\in L_1(\mathbb{R}^+),
\end{equation}
and
\begin{equation}\label{Khachatryan52}
0\leq \chi_2(x)\leq\eta \int\limits_0^\infty K(x,t)(G(\tilde{f}(t))-\tilde{f}(t))dt,\quad x\in\mathbb{R}^+.
\end{equation}
On the other hand, taking into account conditions $a),$ \eqref{Khachatryan2} and the statement of Lemma~2 for arbitrary $R>0$ by virtue of Fubini’s theorem, we obtain
$$0\leq \int\limits_0^R \int\limits_0^\infty K(x,t)(G(\tilde{f}(t))-\tilde{f}(t))dtdx=\int\limits_0^\infty (G(\tilde{f}(t))-\tilde{f}(t)) \int\limits_0^R K(x,t)dxdt\leq$$
$$\leq \int\limits_0^\infty (G(\tilde{f}(t))-\tilde{f}(t)) \int\limits_0^\infty K(t,x)dxdt\leq \int\limits_0^\infty (G(\tilde{f}(t))-\tilde{f}(t))dt<+\infty.$$
In the last inequality, letting $R\rightarrow+\infty$ we obtain that
\begin{equation}\label{Khachatryan53}
\int\limits_0^\infty K(x,t)(G(\tilde{f}(t))-\tilde{f}(t))dt\in L_1(\mathbb{R}^+).
\end{equation}
From \eqref{Khachatryan51}, \eqref{Khachatryan52} and \eqref{Khachatryan53} we arrive at inclusions \eqref{Khachatryan49} and \eqref{Khachatryan50}.

Let us now multiply both sides of the obvious equality
$$f^*(x)-\tilde{f}(x)=\int\limits_0^\infty K(x,t)(G(f^*(t))-G(\tilde{f}(t)))dt,\quad x\in\mathbb{R}^+$$
by the function $G(\tilde{f}(x))-\tilde{f}(x)$ and due to \eqref{Khachatryan49}, \eqref{Khachatryan50} we integrate the resulting equality with respect to $x$ onto $(0,+\infty).$ As a result, taking into account condition \eqref{Khachatryan2} and Fubini’s theorem, we will have
$$\int\limits_0^\infty (f^*(x)-\tilde{f}(x))(G(\tilde{f}(x))-\tilde{f}(x))dx=\int\limits_0^\infty (G(\tilde{f}(x))-\tilde{f}(x))\int\limits_0^\infty K(x,t) (G(f^*(t))-G(\tilde{f}(t)))dtdx=$$
$$=\int\limits_0^\infty(G(f^*(t))-G(\tilde{f}(t)))\int\limits_0^\infty K(x,t)(G(\tilde{f}(x))-\tilde{f}(x))dxdt=$$
$$= \int\limits_0^\infty(G(f^*(t))-G(\tilde{f}(t))) \int\limits_0^\infty K(t,x)(G(\tilde{f}(x))-\tilde{f}(x))dxdt=$$
$$=\int\limits_0^\infty(G(f^*(t))-G(\tilde{f}(t))) \left(\tilde{f}(t)-\int\limits_0^\infty K(t,x)\tilde{f}(x)dx\right)dt=:\mathcal{D}.$$
From conditions $1)-3)$ it immediately follows that the function $y=G(u)$ has an inverse on the set $\mathbb{R}^+: Q:= G^{-1},$ and $Q(0)=0, Q(\eta)=\eta,$ $Q\in C(\mathbb{R}^+),$ $y=Q(u)\uparrow$ on $\mathbb{R}^+$ and $y=Q(u)$ is convex on the set $\mathbb{R}^+.$ Since $0<G(\tilde{f}(x))<\eta,$ $x\in\mathbb{R}^+$ (because of $0<\tilde{f}(x)<\eta,$ $x\in\mathbb{R}^+$ and $G(0)=0,$ $G(\eta)=\eta,$ $y=G(u) \uparrow$ on $\mathbb{R}^+$), then according to statement
6.14.1 from the book \cite{khardi24} we have
\begin{equation}\label{Khachatryan54}
\begin{array}{c}
\displaystyle\int\limits_0^\infty K(t,x)\tilde{f}(x)dx= \int\limits_0^\infty K(t,x)Q(G(\tilde{f}(x)))dx\geq\\ \displaystyle\geq\int\limits_0^\infty K(t,x)dx Q\left(\frac{\int\limits_0^\infty K(t,x)G(\tilde{f}(x))dx}{\int\limits_0^\infty K(t,x)dx}\right),\quad t\in \mathbb{R}^+.
\end{array}
\end{equation}
On the other hand, one can easily prove that for all $u\in[0,1]$ and $v\geq0$ the following inequality holds:
\begin{equation}\label{Khachatryan55}
uQ(v)\geq Q(uv).
\end{equation}
From \eqref{Khachatryan54} and \eqref{Khachatryan55} due to condition $a)$ we obtain the inequality:
\begin{equation}\label{Khachatryan56}
\int\limits_0^\infty K(t,x)\tilde{f}(x)dx\geq Q\left(\int\limits_0^\infty K(t,x)G(\tilde{f}(x))dx\right)=Q(\tilde{f}(t)),\quad t\in \mathbb{R}^+.
\end{equation}
Therefore, taking into consideration \eqref{Khachatryan56} for the value of $\mathcal{D}$ we obtain the following inequality:
$$\mathcal{D}\leq \int\limits_0^\infty (G(f^*(t))-G(\tilde{f}(t)))(\tilde{f}(t)-Q(\tilde{f}(t)))dt.$$
Thus we arrive at the following inequality:
$$\int\limits_0^\infty (f^*(x)-\tilde{f}(x))(G(\tilde{f}(x))-\tilde{f}(x))dx\leq \int\limits_0^\infty (G(f^*(x))-G(\tilde{f}(x)))(\tilde{f}(x)-Q(\tilde{f}(x)))dx$$
or due to the definition of a measurable set $\Gamma$ and the fact that for $x\in\mathbb{R}\backslash\Gamma,$ $f^*(x)=\tilde{f}(x)$ we have
\small{\begin{equation}\label{Khachatryan57}
\int\limits_{\Gamma} (f^*(x)-\tilde{f}(x))\left(G(\tilde{f}(x))-G(Q(\tilde{f}(x)))-\frac{G(f^*(x))-G(\tilde{f}(x))}{f^*(x)-\tilde{f}(x)}
(\tilde{f}(x)-Q(\tilde{f}(x))) \right)dx\leq0.
\end{equation}}
Since $0<\tilde{f}(x)<\eta,$ $x\in\mathbb{R}^+,$ then from the properties of the function $Q$ listed above it immediately follows that
\begin{equation}\label{Khachatryan58}
Q(\tilde{f}(x))<\tilde{f}(x), \quad x\in\mathbb{R}^+.
\end{equation}
Taking into account \eqref{Khachatryan58} inequality \eqref{Khachatryan57} can be rewritten in the following form
\small{\begin{equation}\label{Khachatryan59}
\int\limits_{\Gamma} (f^*(x)-\tilde{f}(x))(\tilde{f}(x)-Q(\tilde{f}(x)))\left(\frac{G(\tilde{f}(x))-G(Q(\tilde{f}(x)))}{\tilde{f}(x)-Q(\tilde{f}(x))}-
\frac{G(f^*(x))-G(\tilde{f}(x))}{f^*(x)-\tilde{f}(x)}\right)dx\leq0.
\end{equation}}
Let us now note that, due to conditions $1)-3)$ for all points $x$ from the measurable set $\Gamma$ the following strict inequality holds (see Fig. 2):
\begin{equation}\label{Khachatryan60}
\frac{G(\tilde{f}(x))-G(Q(\tilde{f}(x)))}{\tilde{f}(x)-Q(\tilde{f}(x))}>
\frac{G(f^*(x))-G(\tilde{f}(x))}{f^*(x)-\tilde{f}(x)},\quad x\in\Gamma.
\end{equation}
From \eqref{Khachatryan60}, \eqref{Khachatryan48} and \eqref{Khachatryan58} to \eqref{Khachatryan59} we arrive at a contradiction. Therefore $\tilde{f}(x)=f^*(x),$ $x\in\mathbb{R}^+.$ Thus, the theorem is proved.
\end{proof}

\begin{figure}[!h]
  \centering
  \includegraphics[width=5in]{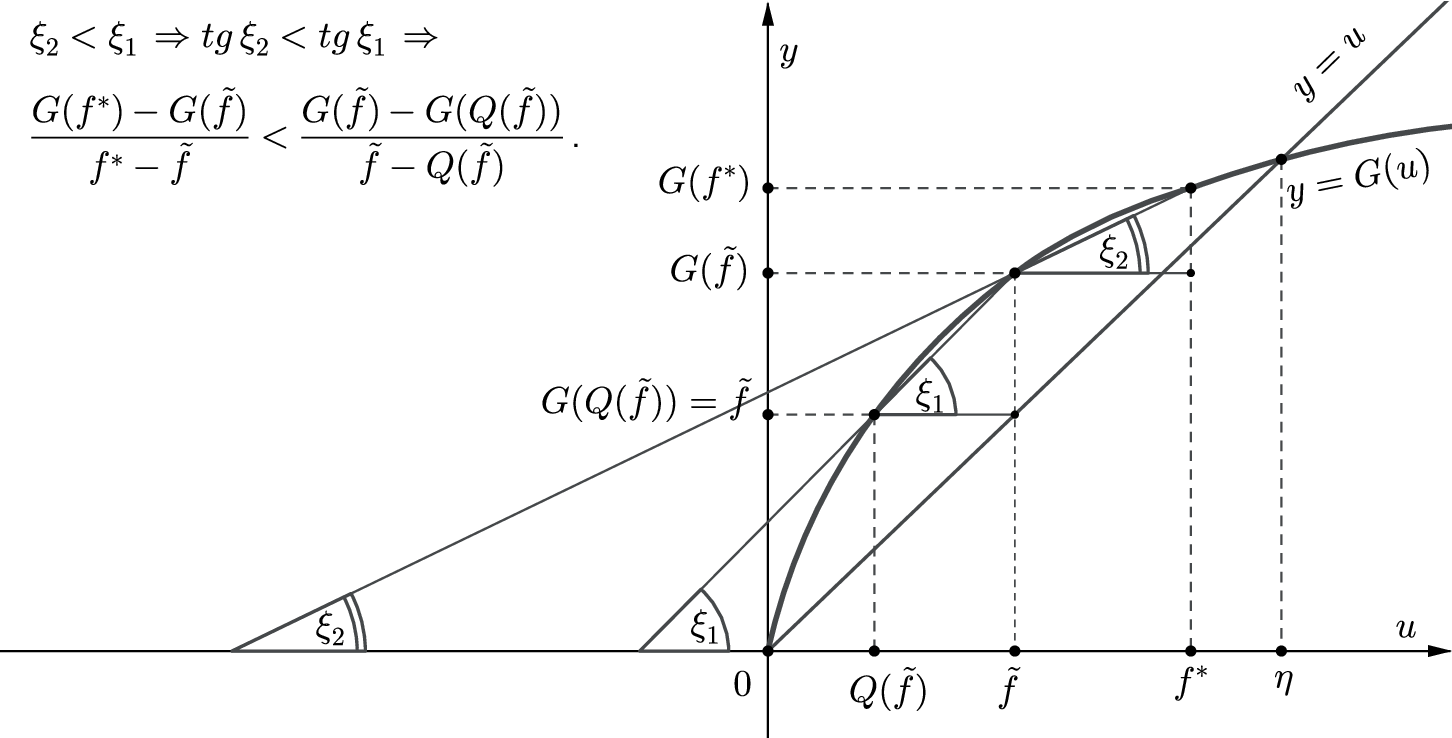}\\
  \caption{}
\end{figure}

\begin{Remark}
It should be noted that Theorem~2 proved above generalizes and complements the uniqueness
theorem of work \cite{petr20} for kernels $K$ depending on the sum and difference of arguments.
\end{Remark}

\section{Summable and bounded solution of one class of
nonlinear integral equations with a monotone
Hammerstein - Nemytsky type operator on the positive
half-line}
Consider the following class of nonlinear integral equations on the semi-axis:
\begin{equation}\label{Khachatryan61}
	\Phi(x)=G_0(x,\Phi(x))+\int\limits_0^\infty K(x,t)G_1(t,\Phi(t))dt,\quad x\in\mathbb{R}^+
\end{equation}
with respect to the unknown measurable non-negative and bounded function $\Phi(x)$ on $\mathbb{R}^+.$ In equation \eqref{Khachatryan61}, kernel $K$ has properties $a)-c),$ and nonlinear functions $G_0$ and $G_1$ defined on the set $\mathbb{R}^+\times\mathbb{R}^+$
satisfy the following \textquotedblleft criticality\textquotedblright condition
\begin{equation}\label{Khachatryan62}
	G_0(x,0)=G_1(x,0)=0,\quad x\in\mathbb{R}^+,
\end{equation}
having the following properties
\begin{enumerate}
	\item [$a_1)$] there exists a number $\xi\in(0,\eta)$ such that
	$$G_0(x,\xi\gamma(x))\geq\xi\gamma(x),\quad x\in\mathbb{R}^+,$$ and
	$$G_0(x,\eta)\leq\eta\gamma(x),\quad x\in\mathbb{R}^+,$$ where the function $\gamma(x)$ is defined in condition $b),$
	\item [$a_2)$] for any fixed $x\in\mathbb{R}^+$ functions $y=G_j(x,u)\uparrow$ in $u$ on the interval $[0,\eta],\, j=0,1.$
	\item [$a_3)$] double inequality takes place
	$$0\leq G_1(x,u)\leq\eta-G(\eta-u),\quad x\in\mathbb{R}^+,\quad u\in[0,\eta],$$ where $G$ satisfies conditions $1)-4).$
	\item [$a_4)$] the functions $y=G_j(x,u),\, j=0,1$ on the set $\mathbb{R}^+\times[0,\eta]$ satisfy the Carathéodory condition with respect to the argument $u,$ i.e. for each fixed $u\in[0,\eta]$ function $G_j(x,u),\, j=0,1$ are measurable in $x$ on $\mathbb{R}^+$ and for almost all $x\in\mathbb{R}^+$ these functions are continuous on the interval $[0,\eta].$
\end{enumerate}
The following theorem holds:
\begin{theorem}
	Under conditions $a)-c), a_1)-a_4),$ and \eqref{Khachatryan62}, equation \eqref{Khachatryan61} has a non-negative non-trivial integrable bounded
	solution on $\mathbb{R}^+$ and $\lim\limits_{x\rightarrow+\infty}\Phi(x)=0.$
\end{theorem}
\begin{proof}
	Let us consider the following successive approximations for the equation \eqref{Khachatryan61}:
	\begin{equation}\label{Khachatryan63}
		\begin{array}{c}
			\displaystyle \Phi_{n+1}(x)=G_0(x,\Phi_n(x))+\int\limits_0^\infty K(x,t)G_1(t,\Phi_n(t))dt,\\
			\displaystyle \Phi_{0}(x)=\xi\gamma(x),\quad n=0,1,2,\ldots,\quad x\in\mathbb{R}^+.
			\displaystyle
		\end{array}
	\end{equation}
	First we prove that
	\begin{equation}\label{Khachatryan64}
		\Phi_n(x)\downarrow\quad\mbox{on}\quad n,\quad x\in\mathbb{R}^+,
	\end{equation}
	\begin{equation}\label{Khachatryan65}
		\Phi_n(x)\leq\eta, \quad n=0,1,2,\ldots,\quad x\in\mathbb{R}^+.
	\end{equation}
	Indeed, we first show that $\Phi_{0}(x)\leq\Phi_{1}(x)\leq\eta,$ $x\in\mathbb{R}^+.$ Given $a_1)-a_3),$ \eqref{Khachatryan62} and $a)$ from \eqref{Khachatryan63}, we have
	\begin{align*}
	&	\Phi_{1}(x)\leq G_0(x,\xi)+\int\limits_0^\infty K(x,t)G_1(t,\xi)dt\\
	&\leq G_0(x,\eta)+\int\limits_0^\infty K(x,t)G_1(t,\eta)dt\\
	&\leq \eta\gamma(x)+\eta \int\limits_0^\infty K(x,t)dt\\
	&=\eta,\quad x\in\mathbb{R}^+
	\end{align*}
	and $$\Phi_{1}(x)\geq G_0(x,\Phi_{0}(x))\geq\xi\gamma(x)=\Phi_{0}(x),\quad x\in\mathbb{R}^+.$$
	Assuming the facts $\Phi_n(x)\geq \Phi_{n-1}(x),$ $x\in\mathbb{R}^+$ and $\Phi_n(x)\leq\eta,$ $x\in\mathbb{R}^+$  for some natural number $n,$ and considering the monotonicity of the
		functions $G_j(x,u)$ with respect to $u$ for $j=0,1,$ and conditions $a_1), a_3), a),$ as well as relation \eqref{Khachatryan62}, from \eqref{Khachatryan63} we obtain $\Phi_{n+1}(x)\geq \Phi_n(x),$ $x\in\mathbb{R}^+,$ moreover,   $$\Phi_{n+1}(x)\leq G_0(x,\eta)+ \int\limits_0^\infty K(x,t)G_1(t,\eta)dt\leq \eta\gamma(x)+\eta(1-\gamma(x))=\eta,\quad x\in\mathbb{R}^+.$$
	
	Let us now verify that a more precise inequality holds
	\begin{equation}\label{Khachatryan66}
		\Phi_n(x)\leq\eta-f^*(x),\quad n=0,1,2,\ldots,\quad x\in\mathbb{R}^+,
	\end{equation}
	where $f^*(x)$ is the only solution of equation \eqref{Khachatryan1} (see Theorems~1 and 2).
	
	For $n=0,$ estimate \eqref{Khachatryan66} is obtained from the following  inequalities
	$$\Phi_{0}(x)=\xi\gamma(x)\leq\eta\gamma(x)\leq \eta\gamma(x)+\int\limits_0^\infty K(x,t)(\eta-G(f^*(t)))dt=\eta-f^*(x),\,  x\in\mathbb{R}^+.$$
	Let us assume that \eqref{Khachatryan66} holds for some $n\in\mathbb{N}.$ Then, considering the conditions $a), a_1), a_2)$ and $b),$ as well as \eqref{Khachatryan1} from \eqref{Khachatryan63} we have
	\begin{align*}
		\Phi_{n+1}(x)&\leq G_0(x,\eta-f^*(x))+ \int\limits_0^\infty K(x,t)G_1(t,\eta-f^*(t))dt\\
	&\leq G_0(x,\eta)+\int\limits_0^\infty K(x,t)(\eta-G(f^*(t)))dt\\
&\leq\eta\gamma(x)+\eta(1-\gamma(x))-\int\limits_0^\infty K(x,t)G(f^*(t))dt =\eta-f^*(x),\quad x\in\mathbb{R}^+.\\
\end{align*}
	Using conditions $a)- c)$ and $a_3)$ as well as $a_{4})$, one can easily verify by induction on $n$ that each function from the
	sequence $\{\Phi_n(x)\}_{n=1}^\infty$ is a measurable function on $\mathbb{R}^+.$ Thus, from \eqref{Khachatryan64} and \eqref{Khachatryan66} it follows that the sequence of measurable functions $\{\Phi_n(x)\}_{n=1}^\infty$ has a pointwise limit when $n\rightarrow\infty$
	$$\lim\limits_{n\rightarrow\infty}\Phi_n(x)=\Phi(x),$$
	where $\Phi(x)$ satisfies the double inequality:
	\begin{equation}\label{Khachatryan67}
		\xi\gamma(x)\leq \Phi(x)\leq\eta-f^*(x),\quad x\in\mathbb{R}^+.
	\end{equation}
	Taking into account condition $a_4)$ and using the limit theorems of B. Levy and M.A. Krasnoselsky (see \cite{kras25}) we conclude that $\Phi(x)$ satisfies equation \eqref{Khachatryan1} almost everywhere on $\mathbb{R}^+.$ From  \eqref{Khachatryan67} and Theorem~1 it follows that
	$$\Phi\in L_1^0(\mathbb{R}^+)\cap L_\infty(\mathbb{R}^+),$$
	$$\Phi(x)\not\equiv0,\quad x\in\mathbb{R}^+.$$
	Thus, the theorem is completely proved.
\end{proof}
\begin{Remark}
	The question of the uniqueness of the constructed solution to equation \eqref{Khachatryan61} in the class of
	nonnegative nontrivial and bounded functions still remains an open problem. One can only prove the
	uniqueness of equation \eqref{Khachatryan61} in the class of non-negative nontrivial and integrable functions on $\mathbb{R}^+$ if we
	additionally assume that $G_0(x,u)$ is concave in $u$ on $\mathbb{R}^+,$ and $G_1(x,u)=\eta-G(\eta-u),$ $u\in\mathbb{R}^+.$ The proof of this fact is carried out by similar reasoning as in \cite{khach26}.
\end{Remark}

\section{Examples}

At the end of the work we will give examples of the kernel $K$ and the nonlinear functions $G, G_0$ and $G_1.$ First, let's give examples of the kernel $K:$
\begin{enumerate}
  \item [A)] $K(x,t)=\mu(x,t)K_0(x-t),$ $(x,t)\in\mathbb{R}^+\times\mathbb{R}^+,$ where $0<\varepsilon_0\leq\mu(x,t)\leq1,$ $\mu(x,t)=\mu(t,x),$ $(x,t)\in\mathbb{R}^+\times\mathbb{R}^+$ ---  is a continuous function on $\mathbb{R}^+\times\mathbb{R}^+$ and
  $$\sup\limits_{t\in\mathbb{R}^+}(1-\mu(x,t))\in L_1^0(\mathbb{R}^+),$$
and the function $K_0$ has the following properties:
  \small$$0< K_0(-t)=K_0(t), t\in\mathbb{R}^+, K_0\in L_1(\mathbb{R})\cap C_M(\mathbb{R}), \int\limits_0^\infty K_0(\tau)d\tau=\frac{1}{2}, \int\limits_0^\infty y K_0(y)dy<+\infty,$$ where $C_M(\mathbb{R})$ is the space of continuous and bounded functions on the set $\mathbb{R},$
  \item [B)] $$K(x,t)=\mu(x,t)(K_0(x-t)-\delta K_0(x+t)),\,(x,t)\in\mathbb{R}^+\times\mathbb{R}^+, $$
   where the functions $\mu$ and $K_0$ satisfy the conditions of example $A),$ while $K_0(t) \downarrow$ on the set $\mathbb{R}^+,$ and $\delta\in(0,1)$ ---is a numerical parameter,
  \item [C)] $$K(x,t)=\displaystyle\frac{\lambda(x)+\lambda(t)}{2}(K_0(x-t)+\varepsilon K_0(x+t)), \, (x,t)\in\mathbb{R}^+\times\mathbb{R}^+,$$ where $\varepsilon\in(0,1)$--- is a numerical parameter, the kernel $K_0$ satisfies the conditions of Example $B),$ and $\lambda\in C(\mathbb{R}^+),$ $1-\lambda\in L_1^0(\mathbb{R}^+)$ and satisfies the following double inequality $$0<d^*:=\inf\limits_{x\in\mathbb{R}^+}\lambda(x)\leq\lambda(x)\leq1,\quad x\in\mathbb{R}^+.$$
\end{enumerate}
Let us now give examples of nonlinear functions $y=G(u);$
\begin{enumerate}
  \item [I)] $G(u)=u^\alpha,\,\, u\in\mathbb{R}^+,\,\, \alpha\in(0,1)$ ---is a numerical parameter,
  \item [II)] $G(u)=\displaystyle\frac{u^{\alpha^*}+u}{2},\,\, u\in\mathbb{R}^+,\,\, \alpha^*\in(0,1)$---is the parameter,
  \item [III)] $G(u)=\displaystyle\frac{u^{\tilde{\alpha}}+u^{\alpha^*}}{2},\,\, \tilde{\alpha},\alpha^*\in(0,1),\,\, \alpha^* >\tilde{\alpha},\,\, u\in\mathbb{R}^+.$
\end{enumerate}
Finally, the examples of functions $G_0$ and $G_1$ are given as
\begin{enumerate}
  \item [$g_1)$] $G_0(x,u)=\displaystyle\frac{2\xi\gamma(x)u}{u+\xi\gamma(x)},\,\, u\in\mathbb{R}^+,\,\,x\in\mathbb{R}^+,\,\,\xi\in\left(0,\frac{\eta}{2}\right),$
  \item [$g_2)$] $G_0(x,u)=\displaystyle\frac{2\xi\gamma(x)u}{u+\xi\gamma(x)}+\varepsilon^*(x)u^2,\,\, u\in\mathbb{R}^+,\,\,x\in\mathbb{R}^+,\,\,\xi\in\left(0,\frac{\eta}{2}\right),$
  
   and $\varepsilon^*(x)$ is a continuous function on $\mathbb{R}^+$ satisfying the double inequality
    $$0\leq \varepsilon^*(x)\leq \frac{(\eta-2\xi)\gamma(x)+\xi\gamma^2(x)}{\eta(\eta+\xi\gamma(x))},\quad x\in\mathbb{R}^+,$$
  \item [$g_3)$] $G_1(x,u)=\eta-G(\eta-u),\quad u\in\mathbb{R}^+,$
  \item [$g_4)$] $G_1(x,u)=\mathcal{L}(x)(\eta-G(\eta-u)), u\in\mathbb{R}^+, x\in\mathbb{R}^+,$ where $\mathcal{L}\in C(\mathbb{R}^+)$ and $0\leq \mathcal{L}(x)\leq1,$ $ x\in\mathbb{R}^+.$
\end{enumerate}
Let us stop in detail on examples $II), III)$ and $C).$ Let's look at example $II).$ Firstly, it is obvious that for
example $II)$ the number $\eta=1,$ $G(0)=0,$ $G(1)=1,$ $G\in C(\mathbb{R}^+).$ Since $G^\prime(u)=0.5(1+\alpha^*u^{\alpha^*-1})>0$ for $u>0,$ and $G^{\prime\prime}(u)=0.5\alpha^*(\alpha^*-1)u^{\alpha^*-2}<0$ for $u>0,$ then $y=G(u) \uparrow$ on $\mathbb{R}^+$ and $y=G(u)$ is concave on $\mathbb{R}^+.$ Let us finally check condition $4).$ For example $II)$ the
inequality in condition $4)$ takes the following form: $\exists \alpha\in(0,1)$ such that
\begin{equation}\label{Khachatryan68}
\sigma^{\alpha^*}u^{\alpha^*}+\sigma u\geq \sigma^{\alpha}u^{\alpha^*}+\sigma^\alpha u,\quad u\in[0,1],\quad \sigma\in(0,1).
\end{equation}
Since $u\in[0,1],$ then to prove inequality \eqref{Khachatryan68} it is enough to check that there exists $\alpha\in(0,1)$ such that
\begin{equation}\label{Khachatryan69}
\frac{\sigma^{\alpha^*}-\sigma^\alpha}{\sigma^\alpha-\sigma}\geq1,\quad \sigma\in(0,1).
\end{equation}
Taking for example $\alpha=0,5(1+\alpha^*)$ we have
$$\sigma^{\alpha^*}+\sigma\geq 2\sqrt{\sigma^{\alpha^*+1}}=2\sigma^{0,5(\alpha^*+1)}=2\sigma^\alpha,\quad \sigma\in(0, 1).$$
Thus we get
$$\sigma^{\alpha^*}-\sigma^\alpha\geq \sigma^\alpha-\sigma,\quad \sigma\in(0,1),$$
from which follows \eqref{Khachatryan69}. Let us now move on to example $III).$ Obviously, for example $III)$ $G(0)=0,$ $G(1)=1,$ $G\in C(\mathbb{R}^+),$  $G^\prime(u)=0.5(\tilde{\alpha}u^{\tilde{\alpha}-1}+\alpha^*u^{\alpha^*-1})>0,$ $u>0,$ $G^{\prime\prime}(u)=0.5(\tilde{\alpha}(\tilde{\alpha}-1)u^{\tilde{\alpha}-2}+\alpha^*(\alpha^*-1)u^{\alpha^*-2})<0,$ $u>0.$ Therefore, conditions $1)-3)$ are satisfied. Here, choosing the number $\alpha=0.5(\tilde{\alpha}+\alpha^*),$ $\tilde{\alpha}<\alpha^*$ we arrive at condition $4).$

Finally, let's discuss example $C).$ It's obvious that
$$K\in C(\mathbb{R}^+\times\mathbb{R}^+)\quad \mbox{and}\quad K(x,t)>0,\quad (x,t)\in\mathbb{R}^+\times\mathbb{R}^+.$$
On one hand using the properties of $K,$ one can calculate that
\begin{align}\label{Khachatryan70}
\int\limits_0^\infty K(x,t)dt& \leq \int\limits_0^\infty (K_0(x-t)+\varepsilon K_0(x+t))dt\\
	&=1-(1-\varepsilon) \int\limits_x^\infty K_0(y)dy\\
	&<1,\quad
\end{align}
for $x\in\mathbb{R}^+.$

On the other hand, by the definition of $K$ we obtain
\begin{align}\label{ineq}
	&\int\limits_0^\infty K(x,t)dt\\
	&=0.5\lambda(x)\left(1-(1-\varepsilon) \int\limits_x^\infty K_0(y)dy\right)+0.5\int\limits_0^\infty (K_0(x-t)+\varepsilon K_0(x+t))\lambda(t)dt,
\end{align}
for $x\in\mathbb{R}^+.$

As $\lim\limits_{x\rightarrow+\infty}\lambda(x)=1,$  using \eqref{Khachatryan70} and \eqref{ineq}, we get to 
\begin{equation}\label{Khachatryan71}
	\lim\limits_{x\rightarrow+\infty}\int\limits_0^\infty K(x,t)dt=1,
\end{equation}
Hence
$$\sup\limits_{x\in\mathbb{R}^+}\int\limits_0^\infty K(x,t)dt=\lim\limits_{x\rightarrow+\infty}\int\limits_0^\infty K(x,t)dt=1.$$
Now, we prove that $\gamma\in L_1(\mathbb{R}^+).$ Based on the definition of $\gamma,$ we have
\begin{align}
&\gamma(x)\notag\\
&=1-0.5\lambda(x)\left(1-(1-\varepsilon) \int\limits_x^\infty K_0(y)dy\right)-0.5\int\limits_0^\infty (K_0(x-t)+\varepsilon K_0(x+t))\lambda(t)dt\notag\\
&\leq 0.5(1-\lambda(x))+0.5(1-\varepsilon)\int\limits_x^\infty K_0(y)dy+\notag\\
&+0.5 \left(\int\limits_0^\infty (K_0(x-t)+ K_0(x+t))dt-\int\limits_0^\infty (K_0(x-t)+\varepsilon K_0(x+t))\lambda(t)dt \right)\notag\\
&\leq 0.5(1-\lambda(x))+0.5(1-\varepsilon)\int\limits_x^\infty K_0(y)dy+0.5 \int\limits_x^\infty K_0(y)dy+\notag \\
&+ 0.5 \int\limits_0^\infty K_0(x-t)(1-\lambda(t))dt\in L_1(\mathbb{R}^+),\notag\end{align}
for $\int\limits_0^\infty y K_0(y)dy<+\infty,$$\quad 1-\lambda\in L_1(\mathbb{R}^+),\quad K_0\in L_1(\mathbb{R}^+).$

The limit relation $\lim\limits_{x\rightarrow+\infty}\gamma(x)=0$ follows immediately from \eqref{Khachatryan71}. From the properties of the kernel $K_0$ we
immediately obtain that
\begin{align*}
	K(x,t)&\leq K_0(x-t)+\varepsilon K_0(x+t)\\
	&\leq (1+\varepsilon)K_0(x-t)\\
	&\leq \lambda^*(t)K^*(x-t),\,\, \text{for}~ (x,t)\in\mathbb{R}^+\times\mathbb{R}^+,
	\end{align*}
where $K^*(x)=(1+\varepsilon)K_0(x)$ for $x\in\mathbb{R},$ and $\lambda^*(t)>1$ for $t\in\mathbb{R}^+,$ as well as $\lambda^*-1\in L_1^0(\mathbb{R}^+),$ is an arbitrary function.

To complete this section, we also give examples of the functions $K_0, \lambda,$ and $\mu$ as well as  $\lambda^*.$ As examples of $K_{o},$ we have
\begin{itemize}
	\item [$\bullet$]
	\[K_0(x)=\frac{1}{\sqrt{\pi}}e^{-x^2},\quad x\in\mathbb{R},\quad\]
	or,
		\item [$\bullet$]
		\[ K_0(x)=\int\limits_a^b e^{-|x|s}d\tilde{\sigma}(s),\quad x\in\mathbb{R},
		\]
\end{itemize}

where $\tilde{\sigma}\uparrow$ on $[a,b),$ $0<a<b\leq+\infty$ and $2\int\limits_a^b\frac{1}{s}d\tilde{\sigma}(s)=1.$

As an illustration of $\lambda,$ we give
\begin{itemize}
\item[$\bullet$]
\[\lambda(x)=1-(1-d^*)e^{-x},\quad x\in\mathbb{R}^+,\]
or, 
\item[$\bullet$]
\[\lambda(x)=1-(1-d^*)(1+x^2)^{-1},\quad x\in\mathbb{R}^+.
\]
\end{itemize}
The example of $\mu$ is as follows
\begin{equation}\label{Khachatryan72}
	\mu(x,t)=\lambda(x)+\lambda(t)-\lambda(x)\lambda(t),\quad (x,t)\in\mathbb{R}^+\times\mathbb{R}^+,
\end{equation}
while here is the example of $\lambda*,$
$$\lambda^*(x)=1+\frac{e^{-x}}{x^l},\quad x\in\mathbb{R}^+,\quad l\in(0,1).$$
It is worth to note that in example \eqref{Khachatryan72}, the condition $\sup\limits_{t\in\mathbb{R}^+}(1-\mu(x,t))\in L_1^0(\mathbb{R}^+)$ holds because $\sup\limits_{t\in\mathbb{R}^+}(1-\mu(x,t))=(1-d^*)^2e^{-x}\in L_1^0(\mathbb{R}^+).$

\section{Acknowledgements}

K. Khachatryan was supported by the Higher Education and Science Committee of the Republic of Armenia, scientific project №~23RL-1A027. Z. Keyshams acknowledges the support by the Higher Education and Science Committee of the Republic of Armenia, scientific project №~10-3/23/2PostDoc-2. M. Mikaeili Nia is grateful for the support from the Higher Education and Science Committee of the Republic of Armenia, scientific project №~10-3/23/2PostDoc-1. The research work by Z. Keyshams and M. Mikaeili Nia is conducted in the framework of the ADVANCE Research Grants provided by the Foundation for Armenian Science and Technology.

\textbf{Contact information}

1) Zahra Keyshams

Yerevan State University, Alex Manoogian 1, Yerevan, 0025, Republic of Armenia

Faculty of Mathematics and Mechanics,

E-mail: zahrakeyshams7@gmail.com\\

2) Khachatur Aghavardovich Khachatryan

Yerevan State University, Alex Manoogian 1, Yerevan, 0025, Republic of Armenia

Faculty of Mathematics and Mechanics,

E-mail: khachatur.khachatryan@ysu.am\\

3) Monire Mikaeili Nia

Yerevan State University, Alex Manoogian 1, Yerevan, 0025, Republic of Armenia

Faculty of Mathematics and Mechanics,

E-mail: moniremikaeili@yahoo.com\\

\end{document}